\documentclass[10pt]{article}

\usepackage{amsmath}    
\usepackage{graphicx}   
\usepackage{verbatim}   
\usepackage{color}      
\usepackage{subfigure}  

\usepackage{amssymb}
\usepackage{amsthm}

\theoremstyle{definition}
\newtheorem{Def}{Definition}[section]

\newtheorem{Prop}[Def]{Proposition}
\newtheorem{Lem}[Def]{Lemma}
\newtheorem{Thm}[Def]{Theorem}
\newtheorem{Cor}[Def]{Corollary}
\newtheorem{Rem}[Def]{Remark}
\newtheorem{Ass}[Def]{Assumption}

\newcommand{\p}{\mathbb{P}}
\newcommand{\e}{\mathbb{E}}
\newcommand{\real}{\mathbb{R}}
\newcommand{\n}{\mathbb{N}}

\usepackage{slashbox}

\allowdisplaybreaks

\begin{document}
\title{$L^p(\Omega)$-Difference of One-Dimensional Stochastic Differential Equations with Discontinuous Drift \footnote{This research has supported by grants of the Japaneses government.}}
\author{Dai Taguchi \footnote{Ritsumeikan University, 1-1-1 Nojihigashi, Kusatsu, Shiga, 525-8577, Japan, email: dai.taguchi.dai@gmail.com }}
\date{}
\maketitle

\begin{abstract}
We consider a one-dimensional stochastic differential equations (SDE) with irregular coefficients.
The purpose of this paper is to estimate the $L^p(\Omega)$-difference of SDEs using the norm of the difference of coefficients, where the discontinuous drift coefficient satisfies a one-sided Lipschitz condition and the diffusion coefficient is bounded, uniformly elliptic and H\"older continuous.
As an application, we consider the stability problem.
\\\\
\textbf{2010 Mathematics Subject Classification}: 58K25; 41A25; 65C30 \\\\
\textbf{Keywords}: 
stochastic differential equation,
irregular coefficients,
stability problem,
Gaussian estimate	
\end{abstract}

\section{Introduction}

Let us consider a one-dimensional stochastic differential equation (SDE)
\begin{align}\label{SDE_1}
X_t=x_0+ \int_0^t b(X_s)ds + \int_0^t \sigma(X_s)dW_s,
\:x_0 \in \real  ,\: t \in [0,T],
\end{align}
where $W:=(W_t)_{0\leq t \leq T}$ is a standard one-dimensional Brownian motion on a probability space $(\Omega, \mathcal{F},\p)$ with a filtration $(\mathcal{F}_t)_{0\leq t \leq T}$ satisfying the usual conditions.
The drift coefficient $b$ and the diffusion coefficient $\sigma$ are Borel-measurable functions from $\real $ into $\real$.
The diffusion process $X:=(X_t)_{0 \leq t \leq T}$ is used in many fields of application, for example, mathematical finance, optimal control problem and filtering.

Let $X^{(n)}$ be a solution of the SDE (\ref{SDE_1}) with drift coefficient $b_n$ and diffusion coefficient $\sigma_n$.
Basically, the ``stability problem" is a convergence problem such that the sequence $(X^{(n)})_{n \in \n}$ tends to $X$ under the condition of convergence of the coefficients $(b_n, \sigma_n) \to (b, \sigma)$ in some sense.
Stroock and Varadhan in chapter $11$ of	\cite{SV} introduce the stability problem in the law sense in order to consider the martingale problem with continuous and locally bounded coefficients.
Kawabata and Yamada in \cite{KaYa} consider the strong convergence of the stability problem under the condition that the drift coefficients $b$ and $b_n$ are Lipschitz continuous function, diffusion coefficients $\sigma$ and $\sigma_n$ are H\"older continuous and $(b_n, \sigma_n)$ locally uniformly converge to $(b, \sigma)$ (see \cite{KaYa}, example $1$).
Kaneko and Nakao \cite{KaNa} prove that if the coefficients $b_n$ and $\sigma_n$ are uniformly bounded, $\sigma_n$ is uniformly elliptic and $(b_n, \sigma_n)$ tend to $(b, \sigma)$ in $L^1$ sense, then $X^{(n)}$ converge to $X$ in $L^2$ sense. Moreover they also prove that solution of SDE (\ref{SDE_1}) is constructed by the Euler-Maruyama scheme under the condition that coefficients $b$ and $\sigma$ are continuous and linear growth (see \cite{KaNa}, Theorem D).
Recently, under the Nakao-Le Gall condition, Hashimoto and Tsuchiya \cite{HaTsu} prove that $(X^{(n)})_{n \in \n}$ converges to $X$ in $L^p$ sense with $p \geq 1$ and give the rate of convergence under the condition that $b_n \to b$ and $\sigma_n \to \sigma$ in $L^1$ and $L^2$ sense, respectively.
They use the Yamada and Watanabe approximation technique which was introduced in \cite{Yamada} and some estimates for the local time.

On a related study, the convergence for the Euler-Maruyama scheme with non-Lipschitz coefficients have been studied recently.
Yan \cite{Ya} has proven that if the sets of discontinuous points of $b$ and $\sigma$ are countable, then the Euler-Maruyama scheme weakly convergence to the unique weak solution of SDE.
Kohatsu-Higa, Lejay and Yasuda \cite{KLY} have studied weak approximation error for a one-dimensional SDE with the drift ${\bf 1}_{(-\infty, 0]}(x) - {\bf 1}_{(0, +\infty)}(x)$ and constant diffusion. 
Gy\"ongy and R\'asonyi \cite{Gyongy} give the order of the rate of convergence for a one-dimensional SDE when the drift is the sum of a Lipschitz and a monotone decreasing H\"older continuous function and the diffusion coefficient is H\"older continuous.
In \cite{NT}, Ngo and Taguchi extend their results in some sense.
They prove that for multi-dimensional SDE, if the drift coefficient is a one-sided Lipschitz function and the diffusion coefficient is H\"older continuous, then the Euler-Maruyama scheme convergence in $L^p$ sense.
They also give the order of the rate of convergence.
The estimate of the density of the Euler-Maruyama scheme which is proved by Lemaire and Menozzi in \cite{Lemaire} plays a crucial role in their arguments.

The purpose of this paper is to estimate the expectation of difference of  SDEs using the norm of the difference of coefficients.
More precisely, for another one-dimensional SDE
\begin{align}\label{SDE_hat}
\hat{X}_t=x_0 + \int_0^t \hat{b}(\hat{X}_s)ds + \int_0^t \hat{\sigma}(\hat{X}_s)dW_s,
\end{align}
we will prove the following inequality:
\begin{align*}
\e[\sup_{0\leq t \leq T	}|X_t - \hat{X}_t|] \leq C(||b-\hat{b}||_1 \vee || \sigma - \hat{\sigma} ||_2^2)^{\eta-1/2},
\end{align*}
where $\eta$ is H\"older exponent of the diffusion coefficients, $C$ is a positive constant and $|| \cdot ||_p$ is a $L^p$-norm with respect to some measure which is defined in definition \ref{Def_p} and absolute continuous with respect to Lebesgue measure.
We will also estimate $\e[\sup_{0\leq t \leq T }|X_t - \hat{X}_t|^p]$ for any $p > 1$.
As an application of our main results, we can consider the strong rate of convergence for the stability problem (see section \ref{Application}). 
Meanwhile, in finance, we may apply main results to estimate error of so-called ``calibration".

To obtain our main results, we will use the fact that the density of the SDE
 is bounded above by a Gaussian type bound.
Using this estimate, we can consider a measure of the norm $||\cdot||_p$ where is absolute continuous with respect to Lebesgue measure.
Gaussian type estimate are well known if the coefficients $b$ and $\sigma$ are smooth enough (see, \cite{Ar}, \cite{St} or \cite{Lemaire}).
In this paper we will prove the Gaussian upper bound if the drift coefficient is bounded measurable and the diffusion coefficient is bounded, uniformly elliptic and H\"older continuous.
The idea of the proof is a ``Taylor-like expansion" of the density.
This expansion is also called the ``Parametrix method" which is a method to construct fundamental solutions for parabolic type partial differential equations (see \cite{Fr}).
Bally and Kohatsu-Higa \cite{BK} prove this expansion using a semigroup approach and obtain the density of the solution of SDE with bounded measurable drift coefficient and diffusion coefficient which is bounded, uniformly elliptic and H\"older continuous.

Finally, we note that SDEs with discontinuous drift coefficient have many applications such as mathematical finance \cite{AI, IIKO}, optimal control problems \cite{BSW} and see also \cite{CS, KR}.

This paper is divided as follows:
Section $2$ introduces the definition of class of function which includes the discontinuous functions and main results.
All the proofs are shown in Section $3$.
In Section $4$, we apply the main results to the stability problem.

\section{Main Results}

\subsection{Notations and Assumptions}

We first define the class of functions which includes discontinuous functions.

\begin{Def}\label{OSL_1} 
A function $f :\real  \rightarrow \real $ is called a \emph{one-sided Lipschitz function} if there exists a positive constant $L$ such that for any $x,y \in \real$,
\begin{align*}
(x-y)(f(x)- f(y)) \leq L|x-y|^2.
\end{align*}
Let $\mathcal{L}$ be the class of all one-sided Lipschitz functions.  
\end{Def} 

\begin{Rem}\label{OSL_3}\rm 
By the definition of the class $\mathcal{L}$, if $f, g \in \mathcal{L}$ and $\alpha \geq 0$, then $f + g$, $ \alpha f \in \mathcal{L}$. 
The one-sided Lipschitz property is closely related to the monotonicity condition.
Actually, a monotone decreasing function is a one-sided Lipschitz function.
Moreover, a Lipschitz continuous function is also a one-sided Lipschitz function.
\end{Rem} 

Now we give assumptions for the coefficients $b, \hat{b}, \sigma$ and $\hat{\sigma}$.

\begin{Ass} \label{Ass_1}\rm 
We assume that the coefficients $b, \hat{b}, \sigma$ and $\hat{\sigma}$ satisfy the following conditions:
\begin{itemize}
\item[$A$-(i)]:  $b \in \mathcal{L}$.
\item[$A$-(ii)]: $b$ and $\hat{b}$ are bounded measurable, i.e., there exists $K > 0$ such that 
\begin{align*} 
   \sup_{x \in \real  }\left(|b(x)| \vee |\hat{b}(x)|\right)
   \leq K.
\end{align*}
\item[$A$-(iii)]: $\sigma$ and $\hat{\sigma}$ are $\eta:= 1/2+\alpha$-H\"older continuous with $\alpha \in [0, 1/2]$, i.e., there exists $K > 0$ such that 
\begin{align*} 
  \sup_{x,y \in \real , x\neq y}\left( \frac{|\sigma(x)-\sigma(y)|}{|x-y|^{\eta}} \vee \frac{|\hat{\sigma}(x)-\hat{\sigma}(y)|}{|x-y|^{\eta}}  \right)
 \leq K.
\end{align*}
\item[$A$-(iv)]: $a= \sigma^2$ and $\hat{a}= \hat{\sigma}^2$ are bounded and uniformly elliptic, i.e., there exists $\lambda \geq 1$ such that for any $x \in \real$, 
\begin{align*}
 \lambda^{-1} \leq {a}(x) \leq \lambda
 \text{ and }
 \lambda^{-1} \leq \hat{a}(x) \leq \lambda.
\end{align*}
\item[$A$-($p$)]: For given $p\geq 1$,
\begin{align*}
  \varepsilon_p:=||b-\hat{b}||_p^{p} \vee ||\sigma-\hat{\sigma}||_{2p}^{2p} <1,
\end{align*}
where $|| \cdot ||_p$ is defined in Definition \ref{Def_p}. Moreover, if $\alpha=0$, 
\begin{align*}
  \frac{1}{\log(1/\varepsilon_p)} <1.
\end{align*}
\end{itemize}
\end{Ass} 

\begin{Def} \label{Def_p} 
Let $p \geq 1$. For bounded measurable function $f$, a norm $|| \cdot ||_p$ is defined by
\begin{align*} 
    ||f||_{p}:= \left( \int_{\real } |f(x)|^p e^{-\frac{|x-x_0|^2}{2(8\lambda)T}} dx \right)^{1/p} < \infty.
\end{align*}
\end{Def} 

\begin{Rem}
Assume that $A$-(ii), $A$-(iii) and $A$-(iv) hold. Then SDE (\ref{SDE_1}) and SDE (\ref{SDE_hat}) have unique strong solution (see \cite{Zv}).
\end{Rem}

\subsection{Main Theorems}

Throughout this paper, we use the positive constant $C$.
Unless explicitly stated otherwise, the constant $C$ depends only on $K,L,T,p,\alpha,\lambda$ and $x_0$.
Moreover the constant $C$ may change from line to line.

\begin{Thm} \label{Main_1}\rm 
Assume that Assumption \ref{Ass_1} with $p=1$ holds.
Then
\begin{align*}
\sup_{\tau \in \mathcal{T}}\e[|X_{\tau}-\hat{X}_{\tau}|] 
&\leq \left\{ \begin{array}{ll}
\displaystyle C\varepsilon_1^{2 \alpha/(2\alpha+1)} &\textit{if } \alpha \in (0, 1/2], \\
\displaystyle \frac{C}{\log (1/\varepsilon_1)} &\textit{if } \alpha = 0,
\end{array}\right.
\end{align*}
where $\mathcal{T}$ is the set of all stopping times $\tau \leq T$.
\end{Thm}

\begin{Thm} \label{Main_2}\rm 
Assume that Assumption \ref{Ass_1} with $p=1$ holds.
Then
\begin{align*} 
\e[\sup_{0 \leq t \leq T}|X_{t} - \hat{X_{t}}|] 
&\leq \left\{ \begin{array}{ll}
\displaystyle C\varepsilon_1^ {\alpha} &\textit{if } \alpha \in (0, 1/2], \\
\displaystyle \frac{C}{\sqrt{\log(1/\varepsilon_1)}} &\textit{if } \alpha = 0.
\end{array}\right.
\end{align*}
\end{Thm}

\begin{Thm} \label{Main_3}\rm 
Let $p \geq 2$. Assume that Assumption \ref{Ass_1} with $p$ holds.
Then
\begin{align*} 
\e[\sup_{0 \leq t \leq T}|X_{t} - \hat{X_{t}}|^p] 
&\leq \left\{ \begin{array}{ll}
\displaystyle C \varepsilon_p^{1/2}. &\textit{if } \alpha = 1/2, \\
\displaystyle C\varepsilon_1^{2 \alpha/(2\alpha+1)} &\textit{if } \alpha \in (0, 1/2), \\
\displaystyle \frac{C}{\log(1/\varepsilon_1)} &\textit{if } \alpha = 0.
\end{array}\right.
\end{align*}
\end{Thm}

Using Jensen's inequality, we can extend Theorem \ref{Main_3} as follows.

\begin{Cor} \label{Main_4}\rm 
Let $p \in (1,2)$. Assume that Assumption \ref{Ass_1} with  $2p $ holds.
Then
\begin{align*}
\e[\sup_{0 \leq t \leq T}|X_{t} - \hat{X_{t}}|^p] 
&\leq \left\{ \begin{array}{ll}
\displaystyle C \varepsilon_{2p}^{1/2}. &\textit{if } \alpha = 1/2, \\
\displaystyle C\varepsilon_1^{\alpha/(2\alpha+1)} &\textit{if } \alpha \in (0, 1/2), \\
\displaystyle \frac{C}{\sqrt{\log(1/\varepsilon_1)}} &\textit{if } \alpha = 0.
\end{array}\right.
\end{align*}
\end{Cor}


We extend Theorem \ref{Main_1} for a error of function of bounded variation.
We first recall the definition of a set of function of bounded variation.

\begin{Def} 
For a function $f:\real \to \real$, a function $T_f$ is defined by
\begin{align*}
T_f(x):=\sup \sum_{j=1}^N|f(x_j)-f(x_{j-1})|,
\end{align*}
where the supremum is taken over $N$ and all partitions $-\infty < x_0 < x_1 < \cdots < x_N=x < \infty.$ 
We call $f$ a function of bounded variation, if 
\begin{align*}
V(f):=\lim_{x \rightarrow \infty }T_f(x) < \infty.
\end{align*}
Let $BV$ be the class of all functions of bounded variation.
\end{Def}

For a function of bounded variation, we have the following error estimate.

\begin{Cor} \label{Main_5}\rm 
Assume that Assumption \ref{Ass_1} with $p=1$ holds.
Then for any $g \in BV$ and $r \geq 1 $, 
\begin{align*}
\e[|g(X_{T})-g(\hat{X}_{T})|^r] 
&\leq \left\{ \begin{array}{ll}
\displaystyle V(g)^{r} C \varepsilon_1^{\alpha/(2\alpha+1)} &\textit{if } \alpha \in (0, 1/2], \\
\displaystyle \frac{V(g)^{r} C }{\sqrt{\log (1/\varepsilon_1)}} &\textit{if } \alpha = 0.
\end{array}\right.
\end{align*}	
\end{Cor}

\section{Proof of the main results}

\subsection{Gaussian bound for the density of SDE}
In this section we consider the density estimate for SDE (\ref{SDE_1}).
It is well known that if the coefficients $b$ and $\sigma$ are smooth enough, then the density of the solution of SDE can be bounded above and below by Gaussian densities (see, \cite{Ar}, \cite{St} or \cite{Lemaire}). 
The aim of this section is to prove the upper bound for the densities of SDE (\ref{SDE_1}) with bounded measurable drift coefficient and bounded, uniformly elliptic and H\"older continuous diffusion coefficient.
First we introduce the result of an expansion for the density $p_t(x_0,\cdot)$ of $X_t$.

\begin{Prop} [\cite{BK} Theorem 5.6 or Proposition 6.2] \label{BK_result_1}
Assume that $A$-(ii), $A$-(iii) and $A$-(iv) hold.
We define
\begin{align*}
\hat{I}_{t_0}^{m}(y_0,y_{m+1})&:= \int_0^{t_0}dt_1\cdots \int_0^{t_{m-1}}dt_m \int_{\real^m}dy_1 \cdots dy_m \nonumber\\
&\times \prod_{i=0}^{m-1}\hat{\theta}_{t_i-t_{i+1}}(y_{i+1},y_i)p_{t_i-t_{i+1}}^{y_i}(y_{i+1},y_i)p_{t_m}^{y_m}(y_{m+1},y_m),
\end{align*}
where the functions $\hat{\theta}_t(x,z)$ and $p_t^{z}(x,y)$ are given by
\begin{align*}
\hat{\theta}_{t}(x,z)&:=\frac{a(x)-a(z)}{2}\left\{ \frac{(z-x-b(z)t)^2}{t^2 a^2(z)} - \frac{1}{ta(z)}\right\} 
-(b(x)-b(z)) \frac{(z-x-b(z)t)}{t a(z)},
\end{align*}
and
\begin{align*}
p_t^{z}(x,y):=\frac{e^{-\frac{|y-x-b(z)t|^2}{2a(z)t}}}{\sqrt{2\pi a(z)t}}.
\end{align*}
Then for any $t\in(0,T]$, the density of $X_t$ exists and satisfies the following expansion:
\begin{align*} 
p_t(x_0,y)=p_t^{y}(x_0,y)+\sum_{m=1}^{\infty}\hat{I}_{t}^{m}(y,x_0).
\end{align*}
\end{Prop}

\begin{Rem}\label{Rem_sol_1}
Under the assumptions $A$-(ii), $A$-(iii) with $\eta \in(1/2, 1]$ and $A$-(iv), Fournier and Printems (\cite{FoPr}, Theorem 2.1, 2010) prove that for any $t \in (0,T]$, the density of $X_t$ exists.
\end{Rem}

The following lemma is useful to prove the upper estimate for the density.

\begin{Lem} \label{Lem_A_1}\rm
Assume that $A$-(ii), $A$-(iii) and $A$-(iv) hold. Let $t_0=T$. For any $x,z \in \real$ and $t \in (0,t_0]$,
\begin{align*}
     \left| \hat{\theta}_{t}(x,z) \right| p_{t}^{z}(x,z)
\leq \frac{C_0}{t^{1-\eta/2}} p_{8\lambda} (t,x,z),
\end{align*}
where
\begin{align*}
p_c(t,x,z):=\frac{e^{-\frac{|x-z|^2}{2ct}}}{\sqrt{2\pi ct}}
\end{align*}
and
\begin{align*}
C_0:=&8 K \lambda^{3/2} \exp (t_0 K^2/(4\lambda )-1/2)t_0^{(1-\eta)/2} \\
+& 2^{(3\eta+1)/2}(4+e)K\lambda^{2+\eta/2}\exp(t_0 K^2/(4\lambda)-1-\eta/2).
\end{align*}
\end{Lem}
\begin{proof}
We divide the two parts;\\
First (drift part):
Note that for any $x,z \in \real$ and $t>0$,
\begin{align}\label{simple_esti_1}
\frac{|z-x-{b}(z)t|}{\sqrt{t{a}(z)}} e^{-\frac{1}{2}\frac{|z-x-{b}(z)t|^2}{2{a}(z)t}}
 \leq \sqrt{\frac{2}{e}}.
\end{align}
Since ${b}$ is bounded and ${a}={\sigma}^2$ is bounded and uniformly elliptic, we have from (\ref{simple_esti_1}),
\begin{align*}
     &|{b}(x)-{b}(z)| \left|\frac{z-x-{b}(z)t}{t{a}(z)} \right| p_{t}^{z}(x,z)
\leq 2 K \left|\frac{z-x-{b}(z)t}{t{a}(z)} \right| \frac{e^{-\frac{|z-x-{b}(z)t|^2}{2{a}(z)t}}}{\sqrt{2\pi {a}(z)t}}\\
&\leq \frac{4 K}{\sqrt{e}} \frac{1}{\sqrt{t{a}(z)}} \frac{e^{-\frac{1}{2}\frac{|z-x-{b}(z)t|^2}{2{a}(z)t}}}{\sqrt{2\pi (2{a}(z))t}}
\leq \frac{4 K \sqrt{\lambda/e}}{\sqrt{t}} \frac{e^{-\frac{|z-x-{b}(z)t|^2}{2(2\lambda )t}}}{\sqrt{2\pi (2/\lambda )t}}
\leq \frac{4 K \lambda^{3/2} e^{-1/2}}{\sqrt{t}} \frac{e^{-\frac{|z-x-{b}(z)t|^2}{2(2\lambda )t}}}{\sqrt{2\pi (2\lambda )t}}.
\end{align*}
Using the inequality $|x-y|^2 \geq \frac{1}{2}|x|^2-|y|^2$, we get
\begin{align*}
|{b}(x)-{b}(z)| \left|\frac{z-x-{b}(z)t}{t{a}(z)} \right| p_{t}^{z}(x,z) 
&\leq \frac{4 K \lambda^{3/2} e^{-1/2}\exp(t_0 K^2/(4\lambda ))}{\sqrt{t}} \frac{e^{-\frac{|z-x|^2}{2(4\lambda )t}}}{\sqrt{2\pi (2\lambda )t}} \\
&\leq \frac{8 K \lambda^{3/2} \exp (t_0 K^2/(4\lambda )-1/2)}{\sqrt{t}} p_{8\lambda}(t,x,z).
\end{align*}
Since for any $t \in (0,t_0]$, 
\begin{align*}
\frac{1}{\sqrt{t}} \leq \frac{t_0^{(1-\eta)/2}}{t^{1-\eta/2}},
\end{align*}
we have
\begin{align*}
&|{b}(x)-{b}(z)| \left|\frac{z-x-{b}(z)t}{t{a}(z)} \right| p_{t}^{z}(x,z) 
\leq \frac{8 K \lambda^{3/2} \exp (t_0 K^2/(4\lambda )-1/2)t_0^{(1-\eta)/2}}{t^{1-\eta/2}} p_{8\lambda}(t,x,z).
\end{align*}
This completes the drift part. Now we consider the second part.\\
Second (diffusion part): Note that
\begin{align}\label{a_esti_1}
     & \frac{|{a}(x)-{a}(z)|}{2}\left| \frac{(z-x-{b}(z)t)^2}{t^2 {a}^2(z)} - \frac{1}{t{a}(z)}\right| \ p_{t}^{z}(x,z)\nonumber\\
&\leq  \frac{|{a}(x)-{a}(z)|}{2} \left\{ \left| \frac{(z-x-{b}(z)t)^2}{t^2 {a}^2(z)}\right| + \frac{1}{t{a}(z)} \right\} e^{-\frac{1}{2}\frac{|z-x-{b}(z)t|^2}{2{a}(z)t}} \frac{e^{-\frac{|z-x-{b}(z)t|^2}{2(2{a}(z))t}}}{\sqrt{2\pi {a}(z)t}}.
\end{align}
Since 
\begin{align*}
\frac{|z-x-{b}(z)t|^2}{|t {a}(z)|} e^{-\frac{1}{2}\frac{|z-x-{b}(z)t|^2}{2{a}(z)t}} \leq \frac{4}{e},
\end{align*}
by using the inequality $|x-y|^2 \geq \frac{1}{2}|x|^2-|y|^2$, (\ref{a_esti_1}) is bounded by
\begin{align}
     & \frac{|{a}(x)-{a}(z)|}{2} \left( \frac{4}{e}+1 \right) \frac{1}{t{a}(z)} \frac{e^{-\frac{|z-x-{b}(z)t|^2}{2(2{a}(z))t}}}{\sqrt{2\pi {a}(z)t}} \nonumber \\
&\leq \frac{(4+e)\lambda^2 \exp(t_0 K^2/(4 \lambda))}{2e} \frac{|{a}(x)-{a}(z)|}{t}\frac{e^{-\frac{|z-x|^2}{2(4\lambda)t}}}{\sqrt{2\pi \lambda t}} \label{a_esti_2}.
\end{align}
By H\"older continuity of ${\sigma}$, (\ref{a_esti_2}) is less than
\begin{align}
\frac{(4+e)\lambda^2 \exp(t_0 K^2/(4 \lambda))}{2e} \frac{K|x-z|^{\eta}}{t}\frac{e^{-\frac{|z-x|^2}{2(4\lambda)t}}}{\sqrt{2\pi \lambda t}}. \label{a_esti_3}
\end{align}
Since
\begin{align*}
      \frac{|x-z|^{\eta}}{t^{\eta/2}} e^{-\frac{|x-z|^2}{2(8\lambda)t}}
 \leq \left(\frac{8\lambda \eta}{ e}\right)^{\eta/2},
\end{align*}
(\ref{a_esti_3}) is bounded by
\begin{align*}
    \frac{2^{(3\eta+1)/2}(4+e)K\lambda^{2+\eta/2}\exp(t_0 K^2/(4\lambda)-1-\eta/2)}{t^{1-\eta/2}} p_{8\lambda}(t,x,z).
\end{align*}
Therefore we have 
\begin{align*}
\left| \hat{\theta}_{t}(x,z) \right| p_{t}^{z}(x,z) 
\leq \frac{C_0}{t^{1-\eta/2}} p_{8\lambda}(t,x,z),
\end{align*}
where the constant $C_0$ is given by
\begin{align*}
C_0:=&8 K \lambda^{3/2} \exp (t_0 K^2/(4\lambda )-1/2)t_0^{(1-\eta)/2} \\
+& 2^{(3\eta+1)/2}(4+e)K\lambda^{2+\eta/2}\exp(t_0 K^2/(4\lambda)-1-\eta/2).
\end{align*}
This concludes the proof of the statement.
\end{proof}

Using Lemma \ref{Lem_A_1}, we can prove the density estimate.

\begin{Prop} \label{GB_n_1} 
Under the assumption of Proposition \ref{BK_result_1}, there exist constant  $C\geq 1$ such that for any $y \in \real$ and $t\in (0,T]$, the density $p_t(x_0,\cdot)$ of $X_t$ satisfies the following estimate: 
\begin{align*}
    p_t(x_0,y)
\leq C p_{8\lambda}(t,x_0,y).
\end{align*}
\end{Prop}
\begin{proof}
Let $x_0=y_{m+1}$, $y=y_0$ and $t_0=t$.
From definition of $\hat{I}_{t_0}^{m}$, Lemma \ref{Lem_A_1} and the Chapman-Kolmogorov equation, we have
\begin{align*}
      |\hat{I}_{t_0}^{m}(y_0,y_{m+1})|
&\leq \int_0^{t_0}dt_1\cdots \int_0^{t_{m-1}}dt_m \int_{\real^m}dy_1 \cdots dy_m \nonumber\\
&\times \prod_{i=0}^{m-1} \left| \hat{\theta}_{t_i-t_{i+1}}(y_{i+1},y_i)\right| p_{t_i-t_{i+1}}^{y_i}(y_{i+1},y_i) p_{t_m}^{y_m}(y_{m+1},y_m) \\
&\leq \int_0^{t_0}dt_1\cdots \int_0^{t_{m-1}}dt_m \int_{\real^m}dy_1 \cdots dy_m \nonumber\\
&\times \prod_{i=0}^{m-1}\frac{C_0}{(t_i-t_{i+1})^{1-\eta/2}}
p_{8\lambda}(t_i-t_{i+1},y_{i+1},y_i) \sqrt{8\lambda} e^{K^2 t_0/2} p_{8\lambda}	(t_m,y_{n+1},y_n) \\
&\leq   \int_0^{t_0}dt_1\cdots \int_0^{t_{m-1}}dt_m \prod_{i=0}^{m-1}\frac{C}{(t_i-t_{i+1})^{1-\eta/2}} p_{8\lambda}(t,x_0, y).
\end{align*}
Since $1-\eta/2 \in [1/2, 3/4]$, we have
\begin{align*}
\sum_{m=1}^{\infty} \int_0^{t_0}dt_1\cdots \int_0^{t_{m-1}}dt_m \prod_{i=0}^{m-1} \frac{C}{(t_i-t_{i+1})^{1-\eta/2}}
& = \sum_{m=1}^{\infty} t_0^{m(1-\eta/2)} C^{m} \prod_{i=0}^{m-1} B(1+i\eta/2, \eta/2) \\
& = \sum_{m=1}^{\infty} \left( t_0^{(1-\eta/2)} C \Gamma(\eta/2) \right)^{m} \frac{1}{\Gamma(1+m\eta/2	)} \\
&<\infty.
\end{align*}
This concludes the statement.
\end{proof}

\subsection{Key estimate}

In this section, we give a key estimate to prove the main results.

\begin{Lem} \label{Key_esti}\rm
Let $p \geq 1$.
Assume that $A$-(ii), $A$-(iii) and $A$-(iv) hold.
Then there exists a constant $C$ such that
\begin{align*}
    \int_0^T \e[|b (\hat{X}_s) - \hat{b} (\hat{X}_s)|^p] ds
&\leq C|| b-\hat{b} ||_p^p
\end{align*}
and
\begin{align*}
    \int_0^T \e[|\sigma (\hat{X}_s) - \hat{\sigma} (\hat{X}_s)|^{2p}] ds
&\leq C|| \sigma-\hat{\sigma} ||_{2p}^{2p}.
\end{align*}
\end{Lem}
\begin{proof}
We only prove the statement for the drift case.
From Proposition \ref{GB_n_1}, there exists $C\geq 1$ such that for any $x \in \real $ and $s \in (0,T]$,
\begin{align*}
    \hat{p}_s(x_0,x)
\leq C p_{8\lambda}(s,x_0,x)
\leq \frac{C}{\sqrt{s}} e^{-\frac{|x-x_0|^2}{2(8\lambda)T}},
\end{align*}
where $\hat{p}_s(x_0,\cdot)$ is a density function of $\hat{X}_s$.
Then 
\begin{align*}
\int_0^T \e[|b (\hat{X}_s) - \hat{b} (\hat{X}_s)|^p] ds
&= \int_0^Tds \int_{\real } dx |b(x) - \hat{b}(x) |^p \hat{p}_s(x_0,x) \\
&\leq \int_0^Tds \frac{C}{\sqrt{s}} \int_{\real } dx |b(x) - \hat{b}(x) |^p e^{-\frac{|x-x_0|^2}{2(8\lambda)T}} \\
&\leq C|| b-\hat{b} ||_p^p
\end{align*}
This concludes the proof.
\end{proof}

\subsection{Yamada and Watanabe approximation technique}
In this section, we introduce the approximation technique of Yamada and Watanabe (see \cite{Yamada} or \cite{Gyongy}).
To prove the main results, we will use this technique.
For each $\delta \in (1,\infty)$ and $\kappa \in (0,1)$, we define a continuous function $\psi _{\delta, \kappa}: \real \to \real^+$ with $supp\: \psi _{\delta, \kappa}  \subset [\kappa/\delta, \kappa]$ such that
\begin{align*} 
\int_{\kappa/\delta}^{\kappa} \psi _{\delta, \kappa}(z) dz = 1 \textit{ and } 0 \leq \psi _{\delta, \kappa}(z) \leq \frac{2}{z \log \delta}, \:\:\:z > 0.
\end{align*}
Our example of $\psi_{\delta, \kappa}$ is
\begin{align*} 
\psi_{\delta, \kappa}(z) := \mu_{\delta, \kappa} \exp \left[{-\frac{1}{(\kappa - z)(z-\kappa /\delta)}}\right] {\bf 1}_{(\kappa /\delta, \kappa)}(z),
\end{align*}
where $\mu_{\delta, \kappa}^{-1} := \int_{\kappa/\delta}^{\kappa} \exp({-\frac{1}{(\kappa - z)(z-\kappa /\delta)}}) dz$.
We define a function $\phi_{\delta, \kappa} \in C^2(\real;\real)$ by
\begin{align*}
\phi_{\delta, \kappa}(x)&:=\int_0^{|x|}\int_0^y \psi _{\delta, \kappa}(z)dzdy.
\end{align*}
It is easy to verify that $\phi_{\delta, \kappa}$ has the following useful properties: 
\begin{align} 
&\frac{\phi'_{\delta, \kappa}(x)}{x}>0, \text{ for any $x \in \real \setminus \{0\}$}. \label{phi1.5}\\
&0 \leq |\phi'_{\delta, \kappa}(x)| \leq 1, \text{ for any $x \in \real$}. \label{phi2} \\
|x| &\leq \kappa + \phi_{\delta, \kappa}(x), \text{ for any $x \in \real $}. \label{phi1}\\ 
\phi''_{\delta, \kappa}(\pm|x|)&=\psi_{\delta, \kappa}(|x|) \leq \frac{2}{|x|\log \delta}{\bf 1}_{[\kappa/\delta, \kappa]}(|x|), \text{ for any $x \in \real \setminus\{0\}$}. \label{phi4}
\end{align}
In particular, the property (\ref{phi1.5}) plays a crucial rule to consider discontinuous drift.

\subsection{Proof of Theorem \ref{Main_1}}
To simplify the discussion, we set
\begin{align*}
Y_t:=X_t-\hat{X_t}.
\end{align*}

\begin{proof}[Proof of Theorem \ref{Main_1}]
Let $\delta \in (1, \infty)$ and $\kappa \in (0,1)$.
From It\^o's formula, (\ref{phi2}) and (\ref{phi1}), we have
\begin{align} \label{Ito_1} 
  |Y_t|
&\leq \kappa + \phi_{\delta, \kappa}(Y_t)\nonumber \\
&=    \kappa + \int_0^t \phi'_{\delta, \kappa}(Y_s) ({b}(X_s) - \hat{b}(\hat{X}_s)) ds + \frac{1}{2}\int_0^t \phi''_{\delta,\kappa}(Y_s) | {\sigma}(X_s) -  \hat{\sigma}(\hat{X}_s) |^2ds
 + M_t^{\delta, \kappa} \nonumber \\
& =    \kappa + \int_0^t \phi'_{\delta, \kappa}(Y_s) ({b}(X_s) - b(\hat{X}_s)) ds 
+ \int_0^t \phi'_{\delta, \kappa}(Y_s) (b(\hat{X}_s) - \hat{b}(\hat{X}_s)) ds \nonumber \\
& + \frac{1}{2}\int_0^t \phi''_{\delta,\kappa}(Y_s) |{\sigma}(X_s) -  \hat{\sigma}(\hat{X}_s) |^2ds
 + M_t^{\delta, \kappa} \nonumber \\
&\leq   \kappa
  + \int_0^t \phi'_{\delta, \kappa}(Y_s) ({b}(X_s) - b(\hat{X}_s)) ds 
  + \int_0^T |b(\hat{X}_s) - \hat{b}(\hat{X}_s)| ds \nonumber \\
& + \frac{1}{2}\int_0^t \phi''_{\delta,\kappa}(Y_s) |{\sigma}(X_s) -  \hat{\sigma}(\hat{X}_s) |^2ds
 + M_t^{\delta, \kappa},
\end{align}
where
\begin{align*}
  M_t^{\delta, \kappa}
:=\int_0^t \phi'_{\delta,\kappa}(Y_s) ( {\sigma}(X_s) -  \hat{\sigma}(\hat{X}_s)) dW_s.
\end{align*}
Note that since $\sigma$, $\hat{\sigma}$ and $ \phi'_{\delta,\kappa}$ are bounded, $(M_t^{\delta, \kappa})_{0\leq t \leq T}$ is a martingale so $\e[M_t^{\delta, \kappa}]=0$.
Since $b \in \mathcal{L}$, for any $x, y \in \real$ with $x \neq y$, we have, from (\ref{phi1.5}) and (\ref{phi2}),
\begin{align*} 
      &\phi'_{\delta, \kappa}(x-y)(b(x)-b(y))
 =\frac{\phi'_{\delta, \kappa}(x-y)}{x-y}(x-y)(b(x)-b(y))
\leq L\frac{\phi'_{\delta, \kappa}(x-y)}{x-y}|x-y|^2
 \leq L|x-y|.
\end{align*} 
Therefore we get
\begin{align}\label{esti_main_1}
 \int_0^t \phi'_{\delta, \kappa}(Y_s) ({b}(X_s) - b(\hat{X}_s)) ds 
\leq  L \int_0^t |Y_s|ds.
\end{align}
Using Lemma \ref{Key_esti} with $p=1$, we have
\begin{align}\label{esti_main_2}
       \int_0^T \e[|b (\hat{X}_s) - \hat{b} (\hat{X}_s)|] ds
&\leq C|| b-\hat{b} ||_1.
\end{align}\label{esti_sigma_1}
From (\ref{phi4}) and $(x+y)^2 \leq 2x^2 + 2y^2$ for any $x,y\geq 0$, we have
\begin{align}
   &\frac{1}{2}\int_0^t \phi''_{\delta,\kappa}(Y_s) |{\sigma}(X_s) -  \hat{\sigma}(\hat{X}_s) |^2ds 
   \leq \int_0^t \frac{{\bf 1}_{[\kappa/\delta, \kappa]}(|Y_s|)}{|Y_s|\log \delta} |{\sigma}(X_s) -  \hat{\sigma}(\hat{X}_s) |^2ds \nonumber\\ 
  & \leq 2\int_0^t \frac{{\bf 1}_{[\kappa/\delta, \kappa]}(|Y_s|)}{|Y_s|\log \delta} |{\sigma}(X_s) -  {\sigma}(\hat{X}_s) |^2ds
  + 2\int_0^t \frac{{\bf 1}_{[\kappa/\delta, \kappa]}(|Y_s|)}{|Y_s|\log \delta} |{\sigma}(\hat{X}_s) -  \hat{\sigma}(\hat{X}_s) |^2ds \nonumber\\
  & \leq 2 \int_0^t \frac{{\bf 1}_{[\kappa/\delta, \kappa]}(|Y_s|)}{|Y_s|\log \delta} |{\sigma}(X_s) -  {\sigma}(\hat{X}_s) |^2ds 
+  \frac{2\delta}{\kappa \log \delta}\int_0^T |{\sigma}(\hat{X}_s) -  \hat{\sigma}(\hat{X}_s) |^2ds \label{esti_main_2_5}
\end{align}
Using Lemma \ref{Key_esti} with $p=1$, we have
\begin{align}\label{esti_main_3}
    \frac{2\delta}{\kappa \log \delta}\int_0^T \e[|{\sigma}(\hat{X}_s) -  \hat{\sigma}(\hat{X}_s) |^2]ds
\leq \frac{C \delta }{\kappa \log \delta} || \sigma-\hat{\sigma} ||_{2}^{2}.
\end{align}
Since $\sigma$ is $\eta=1/2+\alpha$-H\"older continuous, we have
\begin{align}\label{esti_main_4}
2 \int_0^T \frac{{\bf 1}_{[\kappa/\delta, \kappa]}(|Y_s|)}{|Y_s|\log \delta} |{\sigma}(X_s) -  {\sigma}(\hat{X}_s) |^2ds
&\leq 2\int_0^T \frac{{\bf 1}_{[\kappa/\delta, \kappa]}(|Y_s|)}{|Y_s|\log \delta} |Y_s|^{1+2\alpha} ds
 \leq \frac{C \kappa^{2\alpha}}{\log \delta}.
\end{align}
Let $\tau $ be a stopping time with $\tau \leq T$ and $Z_t:=|Y_{t \wedge \tau}|$.
From (\ref{Ito_1}), (\ref{esti_main_1}), (\ref{esti_main_2}), (\ref{esti_main_3}) and (\ref{esti_main_4}), we obtain
\begin{align*}
     \e[Z_t]
&\leq \kappa
 + L\int_0^t \e[Z_s] ds
 + C|| b- \hat{b} ||_1
 + \frac{C \delta }{\kappa \log \delta} || \sigma-\hat{\sigma} ||_{2}^{2}
 + \frac{C \kappa^{2\alpha}}{\log \delta} \\
&\leq  \kappa
 + L\int_0^t \e[Z_s] ds
 + C \varepsilon_1
 +\frac{C \delta }{\kappa \log \delta} \varepsilon_1
 + \frac{C \kappa^{2\alpha}}{\log \delta}.
\end{align*}

Let $\alpha \in (0, 1/2]$. From Assumption \ref{Ass_1} $A$-($1$), $\varepsilon_1 <1$ so we choose $\delta = 2$ and $\kappa = \varepsilon_1^{1/(2\alpha +1) }$. Then we have
\begin{align*}
      \e[Z_t]
 &\leq L\int_0^t \e[Z_s] ds
 + \varepsilon_1^{1/(2\alpha +1)}
 + C\varepsilon_1
 + C \varepsilon_1^{1-1/(2\alpha +1)}
 + C\varepsilon_1^{2 \alpha /(2\alpha +1)} \\
 &\leq L\int_0^t \e[Z_s] ds
 + C\varepsilon_1^{2 \alpha /(2\alpha +1)}.
\end{align*}
By Gronwall's inequality, we get 
\begin{align*}
      \e[Z_t]
\leq  C\varepsilon_1^{2 \alpha /(2\alpha +1)}.
\end{align*}
Therefore by the dominated convergence theorem, we conclude the statement taking $t \to T$.

Let $\alpha = 0$, From Assumption \ref{Ass_1} $A$-($1$), $1/\log (1/\varepsilon_1)<1$ so we choose $\delta = \varepsilon_1^{-1/2}$ and $\kappa = 1/\log (1/\varepsilon_1)$.
Then we have
\begin{align*}
      \e[Z_t]
 &\leq L\int_0^t \e[Z_s] ds
 + \frac{1}{\log (1/\varepsilon_1)}
 + C \varepsilon_1
 + C \varepsilon_1 ^{1/2}
 + \frac{C}{\log(1/\varepsilon_1)}
 \leq L\int_0^t \e[Z_s] ds
 + \frac{C}{\log (1/\varepsilon_1)}.
\end{align*}
By Gronwall's inequality, we obtain
\begin{align*}
      \e[Z_t]
\leq  \frac{C}{\log (1/\varepsilon_1)}.
\end{align*}
Therefore by the dominated convergence theorem, we conclude the statement taking $t \to T$.
\end{proof}

\subsection{Proof of Theorem \ref{Main_2}}
Let $V_t:= \sup_{0 \leq s \leq t} |Y_s|$.
Recall that for each $\delta \in (1,\infty)$ and $\kappa \in (0,1)$,
\begin{align*}
  M_t^{\delta, \kappa}
=\int_0^t \phi'_{\delta,\kappa}(Y_s) ( {\sigma}(X_s) -  \hat{\sigma}(\hat{X}_s)) dW_s.
\end{align*}
The quadratic variation of $M_t^{\delta, \kappa}$ is given by
\begin{align*}
  \langle M^{\delta, \kappa} \rangle_t
 =\int_0^t |\phi'_{\delta,\kappa}(Y_s)|^2 |{\sigma}(X_s) -  \hat{\sigma}(\hat{X}_s)|^2 ds.
\end{align*}

Before proving Theorem \ref{Main_2}, we estimate the expectation of $\sup_{0 \leq s \leq t}|M_s^{\delta, \kappa}|$ for any $t\in [0,T]$, $\delta \in (1,\infty)$ and $\kappa \in (0,1)$. 

\begin{Lem}\label{Mart_L1}
Assume that $A$-(ii), $A$-(iii), $A$-(iv) and $A$-($1$) hold.
Then there exists a constant $C$ such that for any $t\in [0,T]$, $\delta \in (1,\infty)$ and $\kappa \in (0,1)$,
\begin{align*} 
\e[\sup_{0 \leq s \leq t}|M_s^{\delta, \kappa}|] 
&\leq \left\{ \begin{array}{ll}
\displaystyle \frac{1}{2}\e[V_t]
 + C \varepsilon_1^{2 \alpha /(2\alpha +1)}
 + C || \sigma - \hat{\sigma} ||_2 &\textit{if } \alpha \in (0, 1/2], \\
\displaystyle \frac{C}{\sqrt{\log(1/\varepsilon_1)}} &\textit{if } \alpha = 0.
\end{array}\right.
\end{align*}
\end{Lem}
\begin{proof}
From the Burkholder-Davis-Gundy's inequality, there exists a positive constant $C$ such that
\begin{align*}
      \e[\sup_{0 \leq s \leq t}|M_s^{\delta, \kappa}|]
 &\leq C\e[\langle M^{\delta, \kappa}\rangle_t^{1/2}]
 = C\e\left[\left(\int_0^t |{\sigma}(X_s) -  \hat{\sigma}(\hat{X}_s)|^2 ds \right)^{1/2} \right]\\
&\leq C\e\left[\left(\int_0^t |{\sigma}(X_s) -  {\sigma}(\hat{X}_s)|^2 ds \right)^{1/2} \right]
+ C\e\left[\left(\int_0^T |{\sigma}(\hat{X}_s) - \hat{\sigma}(\hat{X}_s)|^2 ds \right)^{1/2} \right].
\end{align*}
From Jensen's inequality and Lemma \ref{Key_esti}, we have	
\begin{align*}
       C\e\left[\left(\int_0^T |{\sigma}(\hat{X}_s) - \hat{\sigma}(\hat{X}_s)|^2 ds \right)^{1/2} \right]
& \leq C\left(\int_0^T \e\left[ |{\sigma}(\hat{X}_s) - \hat{\sigma}(\hat{X}_s)|^2  \right] ds\right)^{1/2} \leq C || \sigma - \hat{\sigma} ||_2.
\end{align*}
Since $\sigma$ is $1/2+\alpha$-H\"older continuous
\begin{align}\label{esti_mart_sig_1}
      \e[\sup_{0 \leq s \leq t}|M_s^{\delta, \kappa}|] 
&\leq C\e\left[\left(\int_0^t | Y_s|^{1+2\alpha} ds \right)^{1/2} \right]
+C || \sigma - \hat{\sigma} ||_2.
\end{align}

If $\alpha \in (0,1/2]$, then we get
\begin{align*}
 C\e\left[\left(\int_0^t | Y_s|^{1+2\alpha} ds \right)^{1/2} \right]
 \leq C\e\left[V_t^{1/2} \left(\int_0^t | Y_s|^{2\alpha} ds \right)^{1/2} \right].
\end{align*}
Using the Young's inequality $xy \leq \frac{x^2}{2C} + \frac{Cy^2}{2}$ for any $x,y\geq 0$ and the Jensen's inequality, we obtain
\begin{align*} 
 C\e\left[\left(\int_0^t | Y_s|^{1+2\alpha} ds \right)^{1/2} \right]
& \leq \frac{1}{2}\e[V_t] + C\int_0^T \e[|Y_s|^{2\alpha}]ds
  \leq \frac{1}{2}\e[V_t] + C\left(\int_0^T \e[|Y_s|]ds \right)^{2\alpha}. \nonumber
\end{align*}
From Theorem \ref{Main_1} with $\tau=s$, we have
\begin{align}\label{esti_mart_sig_3}
C\e\left[\left(\int_0^t | Y_s|^{1+2\alpha} ds \right)^{1/2} \right]
& \leq \frac{1}{2}\e[V_t] + C \varepsilon_1^{2 \alpha /(2\alpha +1)}.
\end{align}
Therefore from (\ref{esti_mart_sig_1}) and (\ref{esti_mart_sig_3}), we get
\begin{align*}
     \e[\sup_{0 \leq s \leq t}|M_s^{\delta, \kappa}|] 
\leq \frac{1}{2}\e[V_t]
 + C \varepsilon_1^{2 \alpha /(2\alpha +1)}
 + C || \sigma - \hat{\sigma} ||_2,
\end{align*}
which concludes the statement for $\alpha \in (0,1/2]$.

If $\alpha =0$, then from Jensen's inequality and Theorem \ref{Main_1} with $\tau=s$, we get
\begin{align*}
C\e\left[\left(\int_0^t | Y_s| ds \right)^{1/2} \right]
\leq  C\left(\int_0^T \e[| Y_s| ] ds\right)^{1/2}
\leq \frac{C}{\sqrt{\log(1/\varepsilon_1)}}.
\end{align*}
Therefore we have
\begin{align*}
\e[\sup_{0 \leq s \leq T}|M_s^{\delta, \kappa}|] 
\leq \frac{C}{\sqrt{\log(1/\varepsilon_1)}} + C || \sigma - \hat{\sigma} ||_2
\leq \frac{C}{\sqrt{\log(1/\varepsilon_1)}}.
\end{align*}
This concludes the statement for $\alpha=0$.
\end{proof}

Using the above estimate, we can prove Theorem \ref{Main_2}.

\begin{proof}[Proof of Theorem \ref{Main_2}]
From (\ref{Ito_1}), (\ref{esti_main_1}), (\ref{esti_main_2_5}), (\ref{esti_main_3}) and (\ref{esti_main_4}), we have
\begin{align}\label{esti_V_1}
  V_t
&\leq  \kappa
 + L \int_0^t V_s ds
 + \int_0^T |{b}(\hat{X}_s) - \hat{b}(\hat{X}_s)| ds \nonumber \\
&+ \frac{2\delta}{\kappa \log \delta} \int_0^T | {\sigma}(\hat{X}_s) -  \hat{\sigma}(\hat{X}_s)|^2ds 
 + \frac{C\kappa^{2\alpha}}{\log \delta}
 + \sup_{0\leq s \leq t} |M_s^{\delta, \kappa}|.
\end{align}

Let $\alpha \in (0, 1/2]$.
From (\ref{esti_V_1}), Lemma \ref{Key_esti} and Lemma \ref{Mart_L1}, we have
\begin{align*}
 \e[V_t]
 &	\leq \kappa
 + L\int_0^t \e[V_s] ds
 + C|| b - \hat{b} ||_1
 + \frac{C \delta }{\kappa \log \delta} || \sigma - \hat{\sigma} ||_2^2
 + \frac{C \kappa^{2\alpha}}{\log \delta} \\
 &+ \frac{1}{2} \e[V_t]
 + C \varepsilon_1^{2 \alpha /(2\alpha +1)}
 + C || \sigma - \hat{\sigma} ||_2\\
 & \leq \kappa
 + L\int_0^t \e[V_s] ds
 + C \varepsilon_1^{1/2}
 + \frac{C \delta }{\kappa \log \delta} \varepsilon_1
 + \frac{C \kappa^{2\alpha}}{\log \delta}
 + \frac{1}{2} \e[V_t]
 + C \varepsilon_1^{2 \alpha /(2\alpha +1)}.
\end{align*}
Hence we get
\begin{align*}
 \e[V_t]
 & \leq 2\kappa
 + 2L\int_0^t \e[V_s] ds
 + C \varepsilon_1^{1/2}
 + \frac{C \delta }{\kappa \log \delta} \varepsilon_1
 + \frac{C \kappa^{2\alpha}}{\log \delta}
 + C \varepsilon_1^{2 \alpha /(2\alpha +1)}.
\end{align*}
Note that $\alpha \leq 2 \alpha /(2\alpha +1) \leq 1/2$.
Taking $\delta =2$ and $\kappa=\varepsilon_1^{1/2}$, we have
\begin{align*}
 \e[V_t]
 &\leq 2L\int_0^t \e[V_s] ds
 + C \varepsilon_1^{1/2}
 + C \varepsilon_1^{\alpha}
 + C \varepsilon_1^{2 \alpha /(2\alpha +1)}
 \leq 2L\int_0^t \e[V_s] ds
 + C \varepsilon_1^{\alpha}.
\end{align*}
By Gronwall's inequality, we obtain
\begin{align*}
 \e[V_t]
 \leq C \varepsilon_1^{\alpha}.
\end{align*}

Let $\alpha = 0$.
From (\ref{esti_V_1}), Lemma \ref{Key_esti} and Lemma \ref{Mart_L1}, we have
\begin{align*}
\e[V_t]
 &	\leq \kappa
 + L\int_0^t \e[V_s] ds
 + C\varepsilon_1
 + \frac{C \delta }{\kappa \log \delta} \varepsilon_1
 + \frac{C}{\log \delta} 
 + \frac{C}{\sqrt{\log(1/\varepsilon_1)}}
\end{align*}
Taking $\delta = \varepsilon_1^{-1/2}$ and $\kappa = 1/\log (1/\varepsilon_1)$, we have
\begin{align*}
 \e[V_t]
 \leq L\int_0^t \e[V_s] ds
 + \frac{C}{\sqrt{\log(1/\varepsilon_1)}}.
\end{align*}
By Gronwall's inequality, we obtain
\begin{align*}
 \e[V_t]
 \leq \frac{C}{\sqrt{\log(1/\varepsilon_1)}}.
\end{align*}
Hence we conclude the proof of Theorem \ref{Main_2}.
\end{proof}

\subsection{Proof of Theorem \ref{Main_3}}
In this section, we also estimate the expectation of $\sup_{0 \leq s \leq t}|M_s^{\delta, \kappa}|^p$ for any $p \geq 2$, $t\in [0,T]$, $\delta \in (1,\infty)$ and $\kappa \in (0,1)$. 

\begin{Lem}\label{Mart_Lp}
Let $p \geq 2$. Assume that $A$-(ii), $A$-(iii), $A$-(iv) and $A$-($p$) hold.
Then there exists a constant $C$ such that for any $t\in [0,T]$, $\delta \in (1,\infty)$ and $\kappa \in (0,1)$,
\begin{align*}
2^{5(p-1)} \e[\sup_{0 \leq s \leq t}|M_s^{\delta, \kappa}|^p] 
&\leq C\e\left[\left(\int_0^t |Y_s|^{1+2\alpha} ds \right)^{p/2} \right] 
  + C|| \sigma - \hat{\sigma} ||_{2p}^{p}.
\end{align*}
In particular, if $\alpha = 1/2$, we have
\begin{align*}
2^{5(p-1)} \e[\sup_{0 \leq s \leq t}|M_s^{\delta, \kappa}|^p] 
\leq \frac{1}{2}\e[V_t^p]
 + C \int_0^t \e[ V_s^{p}] ds
 + C|| \sigma - \hat{\sigma} ||_{2p}^{p}.
\end{align*}
\end{Lem}
\begin{proof}
From the Burkholder-Davis-Gundy's inequality, there exists a positive constant $C$ such that
\begin{align*}
     2^{5(p-1)} &\e[\sup_{0 \leq s \leq t}|M_s^{\delta, \kappa}|^p]
 \leq C\e[\langle M^{\delta, \kappa}\rangle_t^{p/2}]
 \leq C\e\left[\left(\int_0^t | {\sigma}(X_s) - \hat{\sigma}(\hat{X}_s)|^2 ds \right)^{p/2} \right]\\
&\leq C\e\left[\left(\int_0^t | {\sigma}(X_s) - {\sigma}(\hat{X}_s)|^2 ds \right)^{p/2} \right]
+ C\e\left[\left(\int_0^T | {\sigma}(\hat{X}_s) - \hat{\sigma}(\hat{X}_s)|^2 ds \right)^{p/2} \right].
\end{align*}
From Jensen's inequality and Lemma \ref{Key_esti}, we have
\begin{align*}
\e\left[\left(\int_0^T | {\sigma}(\hat{X}_s) - \hat{\sigma}(\hat{X}_s)|^2 ds \right)^{p/2} \right]
\leq C \e\left[ \int_0^T | {\sigma}(\hat{X}_s) - \hat{\sigma}(\hat{X}_s)|^{2p} ds \right]^{1/2} 
 \leq C || \sigma - \hat{\sigma} ||_{2p}^{p}
\end{align*}
Since $\sigma$ is $1/2+\alpha$-H\"older continuous, 
\begin{align*}
     &2^{5(p-1)} \e[\sup_{0 \leq s \leq t}|M_s^{\delta, \kappa}|^p] 
\leq C\e\left[\left(\int_0^t |Y_s|^{1+2\alpha} ds \right)^{p/2} \right] 
  + C|| \sigma - \hat{\sigma} ||_{2p}^{p}
\end{align*}

In particular, if $\alpha = 1/2$, then we get from definition of $V_t$,
\begin{align*}
 C\e\left[\left(\int_0^t | Y_s|^{2} ds \right)^{p/2} \right]
 \leq C\e\left[ \left(V_t \right)^{p/2} \left(\int_0^t	| Y_s|ds \right)^{p/2} \right].
\end{align*}
Using a Young's inequality $xy \leq \frac{x^2}{2C} + \frac{Cy^2}{2}$ for any $x,y\geq 0$ and a Jensen's inequality, we obtain
\begin{align*}
 C\e\left[\left(\int_0^t | Y_s|ds \right)^{p/2} \right]
& \leq \frac{1}{2}\e[V_t^p]
 + C\e\left[ \left(\int_0^t |Y_s|ds\right)^{p} \right] 
 \leq \frac{1}{2}\e[V_t^p]
 + C \int_0^t \e[ V_s^{p}] ds,
\end{align*}
which concludes the statement.
\end{proof}

To prove Theorem \ref{Main_3}, we introduce the following Gronwall type inequality.	
\begin{Lem}[\cite{Gyongy} Lemma 3.2.-(ii)] \label{Lem2_1}\rm 
Let $(A_t)_{0 \leq t \leq T}$ be a nonnegative continuous stochastic process and set $B_t:= \sup_{0\leq s \leq t}A_s$. Assume that for some $r>0$, $q\geq 1$, $\rho \in [1,q]$ and $C_1, \xi \geq 0$, 
\begin{align*} 
\e[B_t^r] \leq C_1 \e\left[ \left( \int_0^t B_s ds \right)^r\right] + C_1 \e\left[ \left( \int_0^t A^{\rho}_s ds \right)^{r/q} \right] +\xi < \infty
\end{align*}
for all $t \in [0,T]$.
If $r\geq q$ or $q+1-\rho < r < q $ hold, then there exists constant $C_2$ depending on $r,q, \rho,T$ and $C_1$ such that
\begin{align*}
\e[B_T^r] \leq C_2 \xi + C_2 \int_0^T\e[A_s]ds.
\end{align*}
\end{Lem}

Now using Lemma \ref{Mart_Lp} and Lemma \ref{Lem2_1}, we can prove Theorem \ref{Main_3}.

\begin{proof}[Proof of Theorem \ref{Main_3}]
From (\ref{esti_V_1}) and the inequality $\left(\sum_{i=1}^m a_i \right)^p \leq 2^{(p-1)(m-1)}\sum_{i=1}^m a_i^p $ for any $p \geq 2$ $a_i>0$ and $m \in \n$, and Jensen's inequality we have
\begin{align*}
  V_t^p
&\leq 2^{5(p-1)} \Bigg( \kappa^p
 + \left(L \int_0^t V_s ds \right)^p
 + T^{p-1}\int_0^T |{b}(\hat{X}_s) - \hat{b}(\hat{X}_s)|^p ds \nonumber \\
&+ \frac{2T^{p-1}\delta^p}{\kappa^p (\log \delta)^p} \int_0^T | {\sigma}(\hat{X}_s) -  \hat{\sigma}(\hat{X}_s)|^{2p}ds 
 + \frac{C\kappa^{2p\alpha}}{(\log \delta)^p}
 + \sup_{0\leq s \leq t} |M_s^{\delta, \kappa}|^p \Bigg).
\end{align*}
From Lemma \ref{Key_esti} with $p \geq 2$, we have
\begin{align*}
 \e[V_t^p]
 &\leq C \kappa^p
 + C\e\left[\left( \int_0^t V_sds \right)^p \right] 
 + C|| b - \hat{b} ||_p^p \\
 &+ \frac{C\delta^p}{\kappa^p (\log \delta)^p} || \sigma - \hat{\sigma} ||_{2p}^{2p}
 + \frac{C\kappa^{2p\alpha}}{(\log \delta)^p}
 + 2^{5(p-1)}\e[\sup_{0\leq s \leq t} |M_s^{\delta, \kappa}|^p].
\end{align*}

If $\alpha = 1/2$,  using Lemma \ref{Mart_Lp}, we have
\begin{align*}
 \e[V_t^p]
 &\leq C \kappa^p
 + C \int_0^t \e[ V_s^{p}] ds
 + C|| b - \hat{b} ||_p^p \\
 &+ \frac{C\delta^p}{\kappa^p (\log \delta)^p} || \sigma - \hat{\sigma} ||_{2p}^{2p} 
 + \frac{C\kappa^{p}}{(\log \delta)^p}
 + \frac{1}{2}\e[V_T^p]
 + C|| \sigma - \hat{\sigma} ||_{2p}^{p}.
\end{align*}
Hence we get
\begin{align*}
 \e[V_t^p]
 &\leq C \kappa^p
 + C \int_0^t \e[ V_s^{p}] ds
 + C|| b - \hat{b} ||_p^p 
 + \frac{C\delta^p}{\kappa^p (\log \delta)^p} || \sigma - \hat{\sigma} ||_{2p}^{2p}
 + \frac{C\kappa^{p}}{(\log \delta)^p}
 + C|| \sigma - \hat{\sigma} ||_{2p}^{p} \\
 &\leq C \kappa^p
 + C \int_0^t \e[ V_s^{p}] ds
 + C \varepsilon_p
 + \frac{C\delta^p}{\kappa^p (\log \delta)^p} \varepsilon_p
 + \frac{C\kappa^{p}}{(\log \delta)^p}
 + C \varepsilon_p^{1/2}.
\end{align*}
Taking $\delta=2$ and $\kappa = \varepsilon_p^{1/(2p)}$, we have
\begin{align*}
 \e[V_t^p]
 \leq C\int_0^t \e[V_s^p] ds
 + C \varepsilon_p^{1/2}.
\end{align*}
By Gronwall's inequality, we obtain
\begin{align*}
 \e[V_t^p]
 \leq C \varepsilon_p^{1/2}.
\end{align*}

If $\alpha \in [0, 1/2)$, using Lemma \ref{Mart_Lp}, we have
\begin{align*}
 \e[V_t^p]
 &\leq C \kappa^p
 + C\e\left[\left( \int_0^t V_sds \right)^p \right] 
 + C|| b - \hat{b} ||_p^p \\
 &+ \frac{C\delta^p}{\kappa^p (\log \delta)^p} || \sigma - \hat{\sigma} ||_{2p}^{2p}
 + \frac{C\kappa^{2p\alpha}}{(\log \delta)^p}
 + C\e\left[\left(\int_0^t |Y_s|^{1+2\alpha} ds \right)^{p/2} \right] 
   + C|| \sigma - \hat{\sigma} ||_{2p}^{p} \\
 & \leq 
   C\e\left[\left( \int_0^t V_sds \right)^p \right] 
 + C\e\left[\left(\int_0^t |Y_s|^{1+2\alpha} ds \right)^{p/2} \right]
 + \kappa^p
 + C \varepsilon_p^{1/2}
 + \frac{C\delta^p}{\kappa^p (\log \delta)^p} \varepsilon_p
 + \frac{C\kappa^{2p\alpha}}{(\log \delta)^p}.
\end{align*}
Using the Lemma \ref{Lem2_1} with $r=p$, $q=2$, $\rho=1+2\alpha$ and
\begin{align*}
 \xi =
 \kappa^p
 + C \varepsilon_p^{1/2}
 + \frac{C\delta^p}{\kappa^p (\log \delta)^p} \varepsilon_p
 + \frac{C\kappa^{2p\alpha}}{(\log \delta)^p},
\end{align*}
we have from Theorem \ref{Main_1} with $\tau=s$, 
\begin{align*}
 \e[V_T^p]
 &\leq 
 C \kappa^p
 + C \varepsilon_p^{1/2}
 + \frac{C\delta^p}{\kappa^p (\log \delta)^p} \varepsilon_p 
 + \frac{C\kappa^{2p\alpha}}{(\log \delta)^p}
 + C \int_0^T \e[|Y_s|] ds\\
 &\leq 
 C \kappa^p
 + C \varepsilon_p^{1/2}
 + \frac{C\delta^p}{\kappa^p (\log \delta)^p} \varepsilon_p
 + \frac{C\kappa^{2p\alpha}}{(\log \delta)^p}
 +\left\{ \begin{array}{ll}
\displaystyle C\varepsilon_1^{2\alpha/(2\alpha+1)} &\textit{if } \alpha \in (0, 1/2), \\
\displaystyle \frac{C}{\log(1/\varepsilon_1)} &\textit{if } \alpha = 0.
\end{array}\right.
\end{align*}
Taking $\delta=2$ and $\kappa = \varepsilon_p^{1/(2p)}$ if $\alpha \in (0, 1/2)$ and $\delta=\varepsilon_1^{-1/2}$ and $\kappa=1/\log(1/\varepsilon_1)$ if $\alpha = 0$, we get
\begin{align*}
 \e[V_T^p]
 \leq \left\{ \begin{array}{ll}
\displaystyle C\varepsilon_1^{2\alpha/(2\alpha+1)} &\textit{if } \alpha \in (0, 1/2), \\
\displaystyle \frac{C}{\log(1/\varepsilon_1)} &\textit{if } \alpha = 0.
\end{array}\right.
\end{align*}
Hence we conclude the proof of Theorem \ref{Main_3}.
\end{proof}

\subsection{Proof of Corollary \ref{Main_5}}

To prove Corollary \ref{Main_5}, we introduce the upper bound for $\e[| g(X) - g(\hat{X})|^r]$ where $g$ is a function of bounded variation, $r \geq 1$, $X$ and $\hat{X}$ are random variables.

\begin{Prop}[\cite{Av}, Theorem 4.3] \label{Avi_ineq}\rm 
Let $X$ and $\hat{X}$ be random variables.
Assume that $X$ has a bounded density $p_X$.
If $g \in BV$ and $r \geq 1 $, then for every $q \geq 1 $, we have
\begin{align*}
\e[| g(X) - g(\hat{X})|^r] \leq 3^{r+1}V(g)^r \left(\sup_{x \in \real} p_X(x)\right)^{\frac{q}{q+1}} \e[|X-\hat{X}|^q]^{1/(q+1)}.
\end{align*}
\end{Prop} 

Using the above proposition, we can prove Corollary \ref{Main_5}.

\begin{proof}[Proof of Corollary \ref{Main_5}]
From Proposition \ref{GB_n_1}, the density $p_T(x_0,\cdot)$ of $X_T$ satisfies the Gaussian upper bound, i.e., there exists a positive constant $C$ such that for any $y \in \real$, 
\begin{align*}
p_T(x_0,y) \leq C p_{8\lambda}(T, x_0, y) \leq \frac{C}{\sqrt{2\pi (8 \lambda)T}}.
\end{align*}
This means that the density $p_T(x_0,\cdot)$ of $X_T$ is bounded. 
Hence from Proposition \ref{Avi_ineq} with $q=1$ and Theorem \ref{Main_1} with $\tau = T$, for any $g \in BV$ and $r \geq 1 $, we have
\begin{align*}
\e[|g(X_{T})-g(\hat{X}_{T})|^r] 
&\leq V(g)^r C \e[|X_T-\hat{X}_T|]^{1/2} 
\leq \left\{ \begin{array}{ll}
\displaystyle V(g)^{r} C \varepsilon_1^{\alpha/(2\alpha+1)} &\textit{if } \alpha \in (0, 1/2], \\
\displaystyle \frac{V(g)^{r} C }{\sqrt{\log (1/\varepsilon_1)}} &\textit{if } \alpha = 0,
\end{array}\right.
\end{align*}
which concludes the proof of statement.
\end{proof}

\section{Application to the stability problem}\label{Application}
In this section, we apply the main results to the stability problem.
For any $n\in \n$, we consider a one-dimensional stochastic differential equations
\begin{align}\label{SDE_n}
X_t^{(n)}=x_0+ \int_0^t b_n(X_s^{(n)})ds + \int_0^t \sigma_n(X_t^{(n)})dW_s.
\end{align}

\begin{Ass}\label{Ass_n}
We assume that the coefficients $b, \sigma$ and the sequence of coefficients $(b_n)_{n \in \n}$ and $(\sigma_n)_{n \in \n}$ satisfy the following conditions:
\begin{itemize}
\item[$A'$-(i)]: $b \in \mathcal{L}$.
\item[$A'$-(ii)]: $b$ and $b_n$ are bounded measurable i.e., there exists $K > 0$ such that
\begin{align*} 
\sup_{n \in \n, x \in \real }\left( |b_n(x)| \vee |b(x)| \right)
\leq K.
\end{align*}
\item[$A'$-(iii)]: $\sigma$ and $\sigma_n$ are $\eta= 1/2+\alpha$-H\"older continuous with $\alpha \in [0, 1/2]$, i.e., there exists $K > 0$ such that 
\begin{align*} 
\sup_{n \in \n, x,y \in \real , x\neq y} \left( \frac{|\sigma(x)-\sigma	(y)|}{|x-y|^{\eta}} \vee \frac{|\sigma_n(x)-\sigma_n(y)|}{|x-y|^{\eta}} \right)
\leq K.
\end{align*}
\item[$A'$-(iv)]: $a=\sigma$ and $a_n:= \sigma_n^2$ are bounded and uniformly elliptic, i.e., there exists $\lambda \geq 1$ such that for any $x \in \real$ and $n \in \n$, 
\begin{align*}
 \lambda^{-1} \leq a(x) \leq \lambda \textit{ and } \lambda^{-1} \leq a_n(x) \leq \lambda.
\end{align*}
\item[$A'$-($p$)]: For given $p>0$,
\begin{align*}
  \varepsilon_{p,n}:=||b-b_n||_p^{p} \vee ||\sigma-\sigma_n||_{2p}^{2p} \to 0
\end{align*}
as $n \to \infty$.
\end{itemize}
\end{Ass}

From the main results Theorem \ref{Main_1}, \ref{Main_2}, \ref{Main_3} and Corollary \ref{Main_4}, \ref{Main_5}, we have the following corollaries.

\begin{Cor} \label{Main_1_n}\rm 
Assume that Assumption \ref{Ass_n} with $p=1$ hold.
Then there exists a positive constant $C$ such that for any $n$ with $\varepsilon_{1,n} <1$, 
\begin{align*}
\sup_{\tau \in \mathcal{T}}\e[|X_{\tau}-X^{(n)}_{\tau}|] 
&\leq \left\{ \begin{array}{ll}
\displaystyle C\varepsilon_{1,n}^{2\alpha/(2\alpha+1)} &\textit{if } \alpha \in (0, 1/2], \\
\displaystyle \frac{C}{\log (1/\varepsilon_{1,n})} &\textit{if } \alpha = 0.
\end{array}\right.
\end{align*}
\end{Cor}

\begin{Cor} \label{Main_2_n}\rm 
Assume that Assumption \ref{Ass_n} with $p=1$ hold.
Then there exists a positive constant $C$ such that for any $n$ with $\varepsilon_{1,n} <1$, 
\begin{align*} 
\e[\sup_{0 \leq t \leq T}|X_t - X_t^{(n)}|] 
&\leq \left\{ \begin{array}{ll}
\displaystyle C\varepsilon_{1,n}^ {\alpha} &\textit{if } \alpha \in (0, 1/2], \\
\displaystyle \frac{C}{\sqrt{\log(1/\varepsilon_{1,n})}} &\textit{if } \alpha = 0.
\end{array}\right.
\end{align*}
\end{Cor}

\begin{Cor} \label{Main_3_n}\rm 
Let $p \geq 2$. Assume that Assumption \ref{Ass_n} with $p$ hold.
Then there exists a positive constant $C$ such that for any $n$ with $\varepsilon_{p,n} <1$, 
\begin{align*} 
\e[\sup_{0 \leq t \leq T}|X_t - X_t^{(n)}|^p] 
&\leq \left\{ \begin{array}{ll}
\displaystyle C \varepsilon_{p,n}^{1/2}. &\textit{if } \alpha = 1/2, \\
\displaystyle C\varepsilon_{1,n}^{2\alpha/(2\alpha+1)} &\textit{if } \alpha \in (0, 1/2), \\
\displaystyle \frac{C}{\log(1/\varepsilon_{1,n})} &\textit{if } \alpha = 0.
\end{array}\right.
\end{align*}
\end{Cor}

\begin{Cor} \label{Main_4_n}\rm 
Let $p \in (1,2)$. Assume that Assumption \ref{Ass_1} and Assumption \ref{Ass_n} with $2p$ hold.
Then there exists a positive constant $C$ such that for any $n$ with $\varepsilon_{2p,n} <1$, 
\begin{align*}
\e[\sup_{0 \leq t \leq T}|X_t - X_t^{(n)}|^p] 
&\leq \left\{ \begin{array}{ll}
\displaystyle C \varepsilon_{2p,n}^{1/2}. &\textit{if } \alpha = 1/2, \\
\displaystyle C\varepsilon_{1,n}^{\alpha/(2\alpha+1)} &\textit{if } \alpha \in (0, 1/2), \\
\displaystyle \frac{C}{\sqrt{\log(1/\varepsilon_{1,n})}} &\textit{if } \alpha = 0.
\end{array}\right.
\end{align*}
\end{Cor}

\begin{Cor} \label{Main_5_n}\rm 
Assume that Assumption \ref{Ass_n} with $p=1$ hold.
Then there exists a positive constant $C$ such that for any $g \in BV$, $r \geq 1 $ and $n$ with $\varepsilon_{1,n} <1$, 
\begin{align*}
\e[|g(X_{T})-g(X^{(n)}_{T})|^r] 
&\leq \left\{ \begin{array}{ll}
\displaystyle V(g)^r C\varepsilon_{1,n}^{\alpha/(2\alpha+1)} &\textit{if } \alpha \in (0, 1/2], \\
\displaystyle \frac{V(g)^r C}{\sqrt{\log (1/\varepsilon_{1,n})}} &\textit{if } \alpha = 0.
\end{array}\right.
\end{align*}
\end{Cor}

The next proposition shows that there exist the sequences $(b_n)_{n\in \n}$ and $(\sigma_n)_{n\in \n}$ satisfying Assumption \ref{Ass_n}.

\begin {Prop} \label{constructions}
(i) Assume $\sup_{x \in \n} |b(x)| \leq K$.
If the set of discontinuity points of $b$ is null a set with respect to the Lebesgue measure, then there exists a differentiable and bounded sequence $(b_n)_{n \in \n}$ such that for any $p \geq 1$,
\begin{align}\label{converge_b_n} 
  \int_{\real} |b(x) - b_n(x) |^p e^{-\frac{|x-x_0|^2}{2(8\lambda)T}} dx \to 0
\end{align}
as $n \to \infty$.
Moreover, if $b$ is a one-sided Lipschitz function, we can construct an explicit sequence $(b_n)_{n\in \n}$ which satisfies a one-sided Lipschitz condition.

(ii) If the diffusion coefficient $\sigma$ satisfies $A'$-(ii) and $A'$-(iii), then there exists a differentiable sequence $(\sigma_n)_{n \in \n}$ such that for any $n\in \n$, $\sigma_n$ satisfies $A'$-(iii), $A'$-(iv) and for any $p \geq 1$,
\begin{align*} 
      \int_{\real} |\sigma(x) - \sigma_n(x) |^{2p} e^{-\frac{|x-x_0|^2}{2(8\lambda)T}} dx \leq \frac{4K^{2p}\sqrt{\pi \lambda T}}{n^{2p \eta}}.
\end{align*}
\end{Prop}
\begin{proof}
Let $\rho(x) := \mu e^{-1/(1-|x|^2)} {\bf 1}({|x|< 1})$ with $\mu^{-1} = \int_{|x|<1} e^{-1/(1-|x|^2)}dx$ and a sequence $(\rho_n)_{n \in \n}$ be defined by $\rho_n(x) := n \rho(nx)$. 
We set $b_n(x):=\int_{\real}b(y)\rho_n(x-y)dy$ and $\sigma_n(x):=\int_{\real}\sigma(y)\rho_n(x-y)dy$.
Then for any $n\in \n$ and $x \in \real$, $ |b_n(x)| \leq K$ and $\lambda^{-1} \leq a_n(x):=\sigma_n^2 (x) \leq \lambda$, $b_n$ and $\sigma_n$ are differentiable.

Proof of (i).
From Jensen's inequality, we have
\begin{align*}
      \int_{\real}|b(x)-b_n(x)|^p e^{-\frac{|x-x_0|^2}{2(8\lambda)T}}dx
&\leq \int_{\real}dx \left( \int_{\real}dy|b(x)-b(y)|\rho_n(x-y) \right) ^pe^{-\frac{|x-x_0|^2}{2(8\lambda)T}}\\
&=     \int_{\real}dx \left( \int_{|z|< 1}dz|b(x)-b(x-z/n)|\rho(z) \right)^pe^{-\frac{|x-x_0|^2}{2(8\lambda)T}}\\
&\leq  \int_{|z|< 1} dz\int_{\real}dx|b(x)-b(x-z/n)|^p e^{-\frac{|x-x_0|^2}{2(8\lambda)T}}\rho(z).
\end{align*}
Since $b$ is bounded, we have
\begin{align}\label{bdd_b_n}
     \int_{\real}|b(x)-b(x-z/n)|^p e^{-\frac{|x-x_0|^2}{2(8\lambda)T}} dx
\leq (2K)^p \int_{\real} e^{-\frac{|x-x_0|^2}{2(8\lambda)T}} dx
\leq 2^{p+2}K^{p} \sqrt{\pi \lambda T}.
\end{align}
Since the set of discontinuity points of $b$ is a null set with respect to the Lebesgue measure, $b$ is continuous almost everywhere. 
From (\ref{bdd_b_n}), using the dominated convergence theorem, we have 
\begin{align*}
      \int_{\real}|b(x)-b(x-z/n)|^p e^{-\frac{|x-x_0|^2}{2(8\lambda)T}}dx \to 0
\end{align*}
as $n \to \infty$.
From this fact and the dominated convergence theorem, $(b_n)_{n\in \n}$ satisfies (\ref{converge_b_n}).
Let $b$ be a one-sided Lipschitz function.
Then, we have
\begin{align*}
      &(x-y) (b_n(x) - b_n(y))
= \int_{\real} (x-y) (b(x-z) - b(y-z)) \rho_n(z)dz \\
&= \int_{\real} \{ (x-z)-(z-y) \} (b(x-z) - b(y-z)) \rho_n(z)dz
\leq L|x-y|^2,
\end{align*}
which implies that $(b_n)_{n\in \n}$ satisfies a one-sided Lipschitz condition.

Proof of (ii).
In the same way as in the proof of (i), we have from H\"older continuity of $\sigma$
\begin{align*}
      &\int_{\real}|\sigma(x)-\sigma_n(x)|^{2p}e^{-\frac{|x-x_0|^2}{2(8\lambda)T}}dx
\leq \int_{|z|< 1} dz\int_{\real}dx|\sigma(x)-\sigma(x-z/n)|^{2p} e^{-\frac{|x-x_0|^2}{2(8\lambda)T}}	\rho(z) \\
&\leq \frac{K^{2p}}{n^{2 p \eta}} \int_{|z|< 1} dz\int_{\real}dx e^{-\frac{|x-x_0|^2}{2(8\lambda)T}}	\rho(z) 
=  \frac{4K^{2p} \sqrt{\pi \lambda T}}{n^{2 p \eta}}.
\end{align*}
Finally, we show that $\sigma_n$ is $\eta$-H\"older continuous.
For any $x, y \in \real$, 
\begin{align*}
      |\sigma_n(x) - \sigma_n(y)|
 \leq \int_{\real} |\sigma(x-z)- \sigma(y-z) |\rho_n (z) dz
 \leq K|x-y|^{\eta},
\end{align*}
which implies that $\sigma_n$ is $\eta$-H\"older continuous.
This concludes that $(\sigma_n)_{n \in \n}$ satisfies (ii).
\end{proof}

\textbf{Acknowledgment.}
The author is very grateful to Professor Arturo Kohatsu-Higa for his supports and fruitful discussions.
The author would also like to thank Hideyuki Tanaka, Tomonori Nakatsu, Libo Li and Takahiro Tsuchiya for their useful advices.
The author would like to express my thanks to Professor Toshio Yamada for his encouragement and comments.
Finally, the author expresses my thanks to our Laboratory members for good discussions.

\end{document}